\documentclass[12pt,a4paper]{article}
\usepackage{url}
\usepackage[utf8]{inputenc}
\usepackage{tikz}
\usepackage{cancel}
\usepackage{color}
\usepackage[normalem]{ulem}
\usepackage[english]{babel}
\usepackage{color}
\usepackage[english]{babel}
\usepackage{color}
\usepackage{graphicx}
\usepackage{amsmath,amssymb,amsthm}
\usepackage[T1]{fontenc}
\usepackage[utf8]{inputenc}
\usepackage{authblk}
\usepackage[english]{babel}
\usepackage{amsmath,amssymb,amsthm}
\usepackage{color}
\usepackage{graphics}
\usepackage{amsmath,amssymb,amsthm}
\usepackage{lipsum}
\usepackage{graphicx}
\usepackage{amsmath,amssymb,bm}
\usepackage{url}

\newtheorem{theorem}{Theorem}

\newtheorem{corollary}[theorem]{Corollary}

\theoremstyle{definition}
\newtheorem{definition}[theorem]{Definition}
\newtheorem{remark}[theorem]{Remark}
\newtheorem{example}{Example}
\newtheorem{procedure}{Procedure}

\renewcommand{\tilde}[1]{\widetilde{#1}}

\title{From primary to dual affine variety codes over the Klein quartic}

\author{Olav Geil
}
\affil{Department of Mathematical Sciences\\ Aalborg University}

\begin{document}
\maketitle

\begin{abstract}
In~\cite{geilozbudak} a novel method was established to estimate the minimum distance of primary affine variety codes and a thorough treatment of the Klein quartic led to the discovery of a family of primary codes with good parameters, the duals of which were originally treated in~\cite{kolluru-feng-rao}[Ex.\ 3.2, Ex.\ 4.1]. In the present work we translate the method from~\cite{geilozbudak} into a method for also dealing with dual codes and we demonstrate that for the considered family of dual affine variety codes from the Klein quartic our method produces much more accurate information than what was found in~\cite{kolluru-feng-rao}. Combining then our knowledge on both primary and dual codes we determine asymmetric quantum codes with desirable parameters. \\


\end{abstract}

\section{Introduction}
\label{intro}
In~\cite{geilozbudak} the authors studied a family of primary affine variety codes over ${\mathbb{F}}_8$, defined from the Klein quartic. The length of these codes is $n=22$, and the dimensions are easy to establish, but to lower bound the minimum distances they introduced a new method where the footprint bound from Gr\"{o}bner basis theory is applied in a novel manner. The resulting codes have good parameters, either similar to the best known codes according to~\cite{grassl} or in a few cases with a defect in the minimum distance of only one. 

In the present paper we show, using simple arguments, how to translate the findings from~\cite{geilozbudak} into information on the corresponding dual affine variety codes. Besides giving us a new family of good classical linear codes this allows us to construct good asymmetric quantum codes, the handling of which requires detailed information on a set of nested classical linear codes $C_2 \subsetneq C_1$ as well as on the set of nested dual codes $C_1^\perp \subsetneq C_2^\perp$. The dual classical codes from the Klein quartic were originally treated in~\cite{kolluru-feng-rao}[Ex.\ 3.2, Ex.\ 4.1] and in~\cite{FR2} which are among the seminal papers on Feng-Rao theory. One way of viewing our method is to consider it as a way of exhuming what (in our understanding) is the most basic principle that makes the Feng-Rao bound work, and to employ this principle in a novel manner. Doing so, for the dual codes related to the Klein quartic we derive much sharper bounds on Hamming weights, and thereby minimum distances, than have previously been reported. In addition to the above we present a universal procedure for establishing primary descriptions of dual affine variety codes and vice versa. Our procedure being universal means that given a polynomial ideal over ${\mathbb{F}}_q$ and a corresponding monomial ordering it returns a primary description for all related dual affine variety codes and vice versa. Thereby it provides a relevant alternative to the newly presented method in~\cite{MR4272610}.

The paper is organized as follows. In Section~\ref{sec2} we recall the method from~\cite{geilozbudak} for handling primary affine variety codes and list results for the case of the Klein quartic that will be needed throughout the paper, including a small refinement which does not change the overall analysis of~\cite{geilozbudak}, but which shall prove important in connection with our treatment of dual codes. Furthermore, we enhance the analysis to also treat the relative distance between two nested codes, the information of which is of importance when constructing asymmetric quantum codes. Then in Section~\ref{sec2point5} we introduce fundamental, yet simple, results which allow us to employ the findings for primary codes to establish bounds on dual codes. This involves descriptions both at a linear code level as well as using the language of affine variety codes. From that  we then in Section~\ref{sec3} establish  extensive information on a family of dual codes from the Klein quartic, and we make the comparison with~\cite{kolluru-feng-rao}[Ex.\ 3.2, Ex.\  4.1], demonstrating the advantage of our method. We then in Section~\ref{sec5} establish primary descriptions of dual codes as well as dual descriptions of primary codes, meaning that for any code considered in~\cite{geilozbudak} as well as any code of the present paper we know both a generator matrix and a parity-check matrix. From this in Section~\ref{sec4.5} we are able to demonstrate tightness of the minimum distance estimates from Section~\ref{sec3} in a considerable amount of cases and to improve upon one of them. Finally, in Section~\ref{sec4} we construct asymmetric quantum codes through the use of the CSS construction, and we demonstrate that they have desirable parameters. This includes examples of impure codes.
\section{Affine variety codes and the results from~\cite{geilozbudak}}\label{sec2}
The concept of affine variety codes was originally coined by Fitzgerald and Lax in~\cite{lax}. Our exposition on the topic relies on the footprint of an ideal
\begin{definition}
Given a field ${\mathbb{F}}$, an ideal $J \subseteq {\mathbb{F}}[X_1, \ldots , X_m]$ and a monomial ordering $\prec$ on the set of monomials in the variables $X_1, \ldots , X_m$, the corresponding footprint is given by
\begin{eqnarray}
\Delta_\prec(J)=\{ M \mid M {\mbox{ is a monomial which is not the leading monomial {\hspace{2cm}}}}  \nonumber \\
{\mbox{ of any polynomial in }} J \}. \nonumber
\end{eqnarray}
\end{definition}
From \cite{clo}[Prop.\ 4, Sec.\ 5.3] we have:
\begin{theorem}\label{thefoot}
The set $\{M+J \mid M \in \Delta_\prec(J)\}$ is a basis for ${\mathbb{F}}[X_1, \ldots , X_m]/J$ as a vector space over ${\mathbb{F}}$.
\end{theorem}
In the following we concentrate on finite fields ${\mathbb{F}}_q$ and extend any given ideal $I \subseteq {\mathbb{F}}_q[X_1, \ldots , X_m]$ to $I_q=I+\langle X_1^q-X_1, \ldots , X_m^q-X_m\rangle$. Clearly, the variety ${\mathbb{V}}(I_q)$ is finite and from its elements $P_1, \ldots , P_n$ we obtain the map $${\mbox{ev}}: {\mathbb{F}}_q[X_1, \ldots , X_m]/I_q \rightarrow {\mathbb{F}}_q^n,$$ $${\mbox{ev}}(F+I_q)=(F(P_1), \ldots , F(P_n)).$$
This map is obviously a vector space homomorphism and it is well-known that it is in fact an isomorphism~\cite{lax}. For simplicity in the following we shall always write ${\mbox{ev}}(F)$ rather than ${\mbox{ev}}(F+I_q)$.\\
For any set of monomials $L \subseteq \Delta_\prec(I_q)$ we now define a linear code $$C(I,L)={\mbox{Span}}_{\mathbb{F}_q}\{{\mbox{ev}}(M) \mid M \in L\}\subseteq {\mathbb{F}}_q^n$$ 
the dimension of which equals $\# L$ due to Theorem~\ref{thefoot} and ${\mbox{ev}}$ being an isomorphism. Such a code is called a primary affine variety code and its dual $C^\perp(I,L)$ is said to be a dual affine variety code. To estimate the minimum distance of $C(I,L)$ we may apply the below corollary of Theorem~\ref{thefoot}, known as the footprint bound~\cite{onorin}.
\begin{corollary}\label{corfoot}
Given an ideal $J \subseteq {\mathbb{F}}_q [X_1, \ldots , X_m]$ the variety ${\mathbb{V}}(J_q)$ is of size $ \# \Delta_\prec (J_q)$.
\end{corollary}
Consider namely a codeword $\vec{c}={\mbox{ev}}(F) \in C(I,L)$ (i.e.\ $F$ is a linear combination of monomials in $L$). Applying the above corollary to the ideal $J=I+\langle F \rangle$ we see that the Hamming weight of $\vec{c}$ equals 
\begin{eqnarray}
w_H(\vec{c}) =n-\# \Delta_\prec ( \langle F \rangle +I_q) =\# \Delta_\prec (I_q) \cap {\mbox{lm}}(\langle F\rangle +I_q)=\# \Box_\prec (F) \label{eqbox}
\end{eqnarray}
where lm denotes the leading monomial and $\Box_\prec (F)=\Delta_\prec (I_q) \cap {\mbox{lm}}(\langle F \rangle +I_q)$.  Knowing a Gr\"{o}bner basis for $\langle F\rangle+I_q$ would provide us with $\# \Box_\prec (F)$, assuming we know $n$, but as we shall need to consider classes of polynomials $F$, rather than individual ones, such a basis cannot be specified. However, if we calculate a Gr\"{o}bner basis $\{H_1(X_1, \ldots , X_m),\ldots ,$ $ H_s(X_1, \ldots , X_m)\}$ for $I_q$ with respect to $\prec$ we obtain full information on $\Delta_\prec(I_q)$ and  
the task then is to estimate how many monomials inside $\Delta_\prec(I_q)$ can be found as a leading monomial of 
a polynomial of the form
\begin{eqnarray}
K(X_1, \ldots , X_m)F(X_1, \ldots  ,X_m)+\sum_{i=1}^s P_i(X_1, \ldots , X_m) H_i(X_1, \ldots , X_m) , \label{eqgrpol}
\end{eqnarray}
where $K, P_1, \ldots , P_s$ are arbitrary polynomials. 

To estimate the minimum distance of a primary affine variety code $C(I,L)$ the most common approach~(e.g.\ \cite{handbook,GeilEvaluationCodes,MR3804810,MR3781399,MR4123876}) is to establish information on $\Box_\prec (F)$ using only information on ${\mbox{lm}}(F)$ and paying in the analysis no attention to the coefficients of lower terms. I.e.\ for each $M \in L$, one detects a set of monomials which is a subset of $\Box_\prec(F)$ for any $F$ having $M$ as leading monomial.
Using in~\cite{impprim} the concept of one-way well-behaving pairs the authors took the first step in the direction of employing information on the coefficients of the non-leading monomials in the (possible) support of $F$. The method from~\cite{geilozbudak} can be seen as a further development in this direction where for each class of polynomials with a given leading monomial ${\mbox{lm}}(F)$, starting from $F$, one applies a series of calculations involving a mix of multiplication by monomials and polynomial divisions, modulo polynomials in $\{F, H_1, \ldots , H_s\}$, the result being in each step a polynomial of the form~(\ref{eqgrpol}).
Writing $\Delta_\prec(I_q)=\{M_1, \ldots , M_n\}$ where the enumeration is done according to the ordering $\prec$ they consider $F=M_i + \sum_{u=1}^{i-1} a_u M_{i-u}$, $a_u \in {\mathbb{F}}_q$. Whenever during the process it is possible to establish conditions on the coefficients for which a substantial amount of monomials in $\Delta_\prec(I_q)$ can be demonstrated to be leading monomials of expressions of the form~(\ref{eqgrpol}) this is recorded and in the following calculations the conditions are assumed {\it{not}} to hold. The process stops when all possible combinations of coefficients $a_1, \ldots , a_{i-1}$ have been covered. We should mention that in continuation of~\cite{geilozbudak} the procedure has also been successfully implemented in~\cite{ps} to treat a family of codes defined from a particular hyperelliptic curve.

We now recall how the above procedure was applied to give a thorough treatment of a family of primary affine variety codes related to the Klein quartic $Y^3+X^3Y+X \in {\mathbb{F}}_8[X,Y]$. The reason for recalling such findings 
is two-fold. Firstly, our method for treating dual codes relies on our findings regarding primary codes, and secondly for the application of asymmetric quantum codes we will need information on both primary and dual codes. Furthermore, for the mentioned application we will need to enhance previous findings on minimum distances to results on relative distances. 

The monomial ordering that we apply is the weighted graded ordering $\prec_w$ defined by $X^{i_1}Y^{j_1} \prec_w X^{i_2}Y^{j_2}$ if either $2i_1+3j_1 < 2i_2+3j_2$ holds or if  $2i_1+3j_1 = 2i_2+3j_2$, but $j_1 <j_2$. 
The Gr\"{o}bner basis for $$I_8=\langle Y^3+X^3Y+X,X^8+X,Y^8+Y\rangle \subseteq {\mathbb{F}}_8[X,Y]$$ becomes $\{Y^3+X^3Y+X,X^8+X,X^7Y+Y\}$ from which the footprint can be seen to equal 
$$\Delta_{\prec_w}(I_8)=\{X^iY^j \mid 0 \leq i \leq 6, 0 \leq j \leq 2\} \cup \{X^7\}$$
corresponding to the fact that the number of affine roots of the Klein curve is 22.

Recall that given $i \in \{ 1, \ldots , 22\}$ the task is for $$F=M_i + \sum_{u=1}^{i-1} a_u M_{i-u}, {\mbox{ \ }} a_u \in {\mathbb{F}}_8$$
to consider an exhaustive series of cases of different combinations of the coefficients $a_1, \ldots , a_{i-1}$. In each case we determine monomials in $\Delta_{\prec_w}(I_8)$ for which a polynomial of the form~(\ref{eqgrpol}) exists having that monomial as leading monomial. It is straightforward to see that among such monomials we have those that are divisible by $M_i$. However, for $8$ particular choices of $M_i$ additional monomials are determined in~\cite{geilozbudak}. Table~\ref{tab1} through Table~\ref{tab8} explain the results for those $8$ cases. 
Here we use the notation 
\begin{eqnarray}
\langle \langle N_1, \ldots , N_u\rangle \rangle = \{M \in \Delta_{\prec_w}(I_8) \mid M {\mbox{ is divisible by some }} N_i, i\in \{1, \ldots , u\}\}\nonumber
\end{eqnarray}
and using this notation for each choice of coefficients we specify in column two monomials which
can be found as leading monomial of a polynomial in~(\ref{eqgrpol}). I.e.\ $\langle \langle N_1, \ldots , N_u \rangle \rangle \subseteq \Box_{\prec_w}(F)$. We stress that the last row in each of the tables is a conclusion that we entirely make for the purpose of treating dual codes in the present paper. What is listed here is the intersection of all the established sets of leading monomials $\langle \langle  \cdots \rangle \rangle$, the information of which being not relevant for the treatment of primary codes. Moreover, in the second column of the tables some entries are marked in bold. These are entries that we are able to add in addition to those established in~\cite{geilozbudak}. The reason for the entries in bold not to be established in~\cite{geilozbudak} is that these entries do not change the analysis for the primary codes, but we will need them to analyze the dual codes. We illustrate our remarks in an example.

\begin{example}
Consider $F(X,Y)=Y^2+a_1X^3+a_2XY+a_3X^2+a_4Y+a_5X+a_6$. Table~\ref{tab3} lists information on $\Box_{\prec_w}(F)$ for $11$ different cases which together cover all possible situations. For the first case $a_1 \notin \{ 1, 0\}$ the entry in the second column reads $\langle \langle Y^2,X^3Y,X^6\rangle \rangle$ meaning that 
\begin{eqnarray}
\{X^6,X^7, X^3Y,X^4Y, X^5Y,X^6Y, Y^2,XY^2,X^2Y^2,X^3Y^2,X^4Y^2,X^5Y^2,X^6Y^2\}\nonumber \\
\subseteq \Box_{\prec_w}(F).\nonumber 
\end{eqnarray}
As the number of elements in the listed set is $13$ we conclude that $w_H({\mbox{ev}}(F) )\geq 13$ whenever $a_1 \notin \{0,1\}$. Ignoring in the next $10$ lines  the entries in bold, we obtain for the considered coefficients information on the corresponding set $\Box_{\prec_w}(F)$ as established in~\cite{geilozbudak}. Observe, that the subsets of $\Box_{\prec_w}(F)$ that we establish are not identical, but that all subsets are of size at least $13$ from which we conclude that for all polynomials $F$ with ${\mbox{lm}}(F)=Y^2$ it holds that $w_H({\mbox{ev}}(F))\geq 13$. The entries in bold (which we add) have no implication for the treatment of the primary code, as for instance we cannot add anything to the first row of Table~\ref{tab3}, and therefore we cannot increase the estimate on the Hamming weight for general polynomial having $Y^2$ as leading monomial beyond $13$. However, for the dual codes we shall need information on the intersection of all $11$ sets listed in column 2 which is the reason for adding the entries in bold. The intersection is what we list in the last line of the table. For the considered choice of leading monomial the intersection is of size $12$ so we would not want to use that as the estimate on the Hamming weight of ${\mbox{ev}}(F)$. Finally, to understand why we can add the values in bold observe for instance that whenever we know that $XY$ is the leading monomial of a polynomial as in~(\ref{eqgrpol}) then from the last line of Table~\ref{tab2} we can conclude that also is $X^5$ (and monomials divisible by it). Similarly when $Y$ is included in column 2, then from Table~\ref{tab1} we see that also $X^4$ (and monomials divisible by it) can be added. 
\end{example}

\begin{table}
\begin{center}
\caption[h]{$Y+a_1X+a_2$}
\label{tab1}
\begin{tabular}{|l|l|}
\hline
Coefficients&Subset of $\Box_{\prec_w}(F)$\\
\hline 
$a_1 \neq 0$&$\langle \langle Y, X^4\rangle  \rangle $\\
\hline
$a_1=0, a_2 \neq 0$&$\langle \langle Y, X^3\rangle\rangle $\\
\hline
$a_1=a_2=0$&$\langle \langle Y,X \rangle \rangle $\\
\hline 
\hline
Intersection & $\langle \langle Y, X^4\rangle \rangle  $\\
\hline
\end{tabular}
\end{center}
\end{table}

\begin{table}
\begin{center}
\caption[h]{$XY+a_1X^2+a_2Y+a_3X+a_4$}
\label{tab2}
\begin{tabular}{|l|l|}
\hline
Coefficients&Subset of $\Box_{\prec_w}(F)$\\
\hline 
$a_1 \neq 0$&$\langle \langle XY, X^5\rangle \rangle$\\
\hline
$a_1=0, a_3 \neq 0$&$\langle \langle XY, X^4\rangle \rangle $\\
\hline
$a_1=a_3=a_4=0$&$\langle \langle XY,X^2 \rangle \rangle$\\
\hline
$a_1=a_3=0, a_4\neq0$&$\langle \langle XY,Y^2,X^5 \rangle \rangle$\\
\hline 
\hline
Intersection & $\langle \langle XY, X^5\rangle \rangle $\\
\hline
\end{tabular}
\end{center}
\end{table}

\begin{table}
\begin{center}
\caption[ht]{$Y^2+a_1X^3+a_2XY+a_3X^2+a_4Y+a_5X+a_6$}
\label{tab3}
\begin{tabular}{|l|l|}
\hline
Coefficients&Subset of $\Box_{\prec_w}(F)$\\
\hline 
$a_1 \notin \{0,1\}$&$\langle \langle Y^2, X^3Y,X^6 \rangle \rangle $\\
\hline
$a_1=0, a_3 \neq 0$&$\langle \langle Y^2, X^3Y, X^5\rangle \rangle $\\
\hline
$a_1=a_3=0, a_5 \neq 0$&$\langle \langle  Y^2,X^3Y,X^4 \rangle \rangle $\\
\hline
$a_1=a_3=a_5=0, a_6\neq0$&$\langle \langle Y^2,X^3 \rangle \rangle $\\
\hline
$a_1=a_3=a_5=a_6=0$&$\langle \langle Y^2,XY, \bm{X^5} \rangle \rangle $\\
\hline
$a_1=1, a_2 \neq 0$&$\langle \langle Y^2,X^4 \rangle \rangle $\\
\hline
$a_1=1, a_2=0, a_3\neq0$&$\langle \langle  Y^2,X^2Y,X^5 \rangle \rangle$\\
\hline
$a_1=1, a_2=a_3=0, a_4 \neq0$&$\langle \langle  Y^2,X^3 \rangle \rangle $\\
\hline
$a_1=1, a_2=a_3=a_4=0, a_5\neq0$&$\langle \langle  Y^2,XY,\bm{X^5} \rangle \rangle$\\
\hline
$a_1=1, a_2=a_3=a_4=a_5=0, a_6 \neq 0$&$\langle \langle Y , \bm{X^4} \rangle \rangle $\\
\hline
$a_1=1, a_2=a_3=a_4=a_5=a_6=0$&$\langle \langle Y^2,X \rangle \rangle $\\
\hline
\hline
Intersection & $\langle \langle Y^2, X^4Y,X^6\rangle \rangle $\\
\hline
\end{tabular}
\end{center}
\end{table}

\begin{table}
\begin{center}
\caption[h]{$X^2Y+a_1Y^2+a_2X^3+a_3XY+a_4X^2+a_5Y+a_6X+a_7$}
\label{tab4}
\begin{tabular}{|l|l|}
\hline
Coefficients&Subset of $\Box_{\prec_w}(F)$\\
\hline 
$a_2 \neq 0$&$\langle \langle X^2Y,X^6 \rangle \rangle $\\
\hline
$a_2=0, a_4 \neq 0$&$\langle \langle X^2Y,X^5\rangle \rangle $\\
\hline
$a_2=a_4=0, a_6 \neq 0$&$\langle \langle X^2Y,X^5 \rangle \rangle$\\
\hline
$a_2=a_4=a_6=0, a_7\neq1$&$\langle \langle  X^2Y,X^5 \rangle \rangle$\\
\hline
$a_2=a_4=a_6=0, a_7=1, a_3 \neq 0$&$\langle \langle X^2Y, {X^5} \rangle \rangle$\\
\hline
$a_2=a_4=a_6=a_3=0, a_7=1, a_5 \neq 0$&$\langle \langle X^2Y,X^5 \rangle \rangle$\\
\hline
$a_2=a_4=a_6=a_3=a_5=0, a_7=1$&$\langle \langle Y,\bm{X^4} \rangle \rangle $\\
\hline
\hline
Intersection & $\langle \langle X^2Y,X^6\rangle \rangle $\\
\hline
\end{tabular}
\end{center}
\end{table}

\begin{table}
\begin{center}
\caption[h]{$XY^2+a_1X^4+a_2X^2Y+a_3Y^2+a_4X^3+a_5XY+a_6X^2+a_7Y+a_8X+a_9$}
\label{tab5}
\begin{tabular}{|l|l|}
\hline
Coefficients&Subset of $\Box_{\prec_w}(F)$\\
\hline 
$a_1 \neq 1, a_1\neq 0$&$\langle \langle  XY^2,X^4Y,X^7\rangle \rangle $\\
\hline
$a_1=0, a_4 \neq 0$&$\langle \langle XY^2,X^4Y,X^6\rangle \rangle $\\
\hline
$a_1=a_4=0, a_6 \neq 0$&$\langle \langle XY^2,X^4Y,X^5 \rangle \rangle$\\
\hline
$a_1=a_4=a_6=0, a_8\neq0$&$\langle \langle XY^2,X^4 \rangle \rangle$\\
\hline
$a_1=a_4=a_6=a_8=0$&$\langle \langle XY^2, X^2Y,\bm{X^6} \rangle \rangle$\\
\hline
$a_1=1, a_2 \neq 0$&$\langle \langle XY^2,X^5 \rangle \rangle$\\
\hline
$a_1=1, a_2=0, a_3\neq a_4$&$\langle \langle XY^2, X^3Y,X^6 \rangle \rangle$\\
\hline
$a_1=1, a_2=0, a_3=a_4, a_5 \neq 0$&$\langle \langle XY^2,X^4 \rangle \rangle $\\
\hline
$a_1=1, a_2=a_5=0, a_3=a_4,  a_6\neq0$&$\langle \langle XY^2,X^2Y,\bm{X^6} \rangle \rangle$\\
\hline
$a_1=1, a_2=a_5=a_6=0, a_3=a_4, a_7 \neq 0$&$\langle \langle Y^2, {X^4} \rangle \rangle $\\
\hline
$a_1=1, a_2=a_5=a_6=a_7=0, a_3=a_4, a_8 \neq 0$&$\langle \langle XY^2,X^2Y,\bm{X^6} \rangle \rangle$\\
\hline
$a_1=1, a_2=a_5=a_6=a_7=a_8=0, a_3=a_4$&$\langle \langle XY^2,X^3 \rangle \rangle $\\
\hline
\hline
Intersection & $\langle \langle XY^2,X^5Y, X^6\rangle \rangle$\\
\hline
\end{tabular}
\end{center}
\end{table}

\begin{table}
\begin{center}
\caption[htbp]{$X^3Y+a_1XY^2+a_2X^4+a_3X^2Y+a_4Y^2+a_5X^3+a_6XY+a_7X^2+a_8Y+a_9X+a_{10}$}
\label{tab6}
\begin{tabular}{|l|l|}
\hline
Coefficients&Subset of $\Box_{\prec_w}(F)$\\
\hline 
An exhaustive set of ten different cases&$\langle \langle  X^3Y,X^7 \rangle \rangle$\\
\hline
\hline
Intersection & $\langle \langle X^3Y,X^7\rangle \rangle $\\
\hline
\end{tabular}
\end{center}
\end{table}

\begin{table*}
\begin{center}
\caption[htbp]{$X^2Y^2+a_1X^5+a_2X^3Y+a_3XY^2+a_4X^4+a_5X^2Y+a_6Y^2+a_7X^3+a_8XY+a_9X^2+a_{10}Y+a_{11}X+a_{12}$}
\label{tab7}
\begin{tabular}{|l|l|}
\hline
Coefficients&Subset of $\Box_{\prec_w}(F)$\\
\hline 
$a_1 \neq 1$&$\langle \langle  X^2Y^2,X^5Y\rangle \rangle  $\\
\hline
$a_1=1, a_2 \neq 0$&$\langle \langle X^2Y^2,X^6\rangle \rangle$\\
\hline
$a_1=1, a_2=0, a_3 \neq a_4$&$\langle \langle  X^2Y^2,X^4Y \rangle \rangle$\\
\hline
$a_1=1, a_2=0, a_3=a_4, a_5\neq0$&$\langle \langle  X^2Y^2,X^5 \rangle\rangle $\\
\hline
$a_1=1, a_2=a_5=0, a_3=a_4, a_6\neq a_7$&$\langle \langle  X^2Y^2, X^3Y,\bm{X^7} \rangle \rangle$\\
\hline
$a_1=1, a_2 =a_5=0, a_3=a_4, a_6=a_7, a_8\neq 0$&$\langle \langle X^2Y^2,X^5 \rangle \rangle $\\
\hline
$a_1=1, a_2=a_5=a_8=0, a_3=a_4, a_6=a_7, a_{9}\neq 0$&$\langle \langle  X^2Y^2, X^3Y,\bm{X^7}  \rangle \rangle$\\
\hline
$a_1=1, a_2=a_5=a_8=a_9=0, a_3=a_4, a_6=a_7, a_{10} \neq 1$&$\langle \langle X^2Y^2,X^5 \rangle \rangle $\\
\hline
$a_1=a_{10}=1, a_2=a_5=a_8=a_9=0, a_3=a_4,  a_6=a_7, $&\\
$a_{11} \neq0$&$\langle \langle  X^2Y^2,X^3Y,\bm{X^7} \rangle \rangle $\\
\hline
$a_1=a_{10}=1, a_2=a_5=a_8=a_9=a_{11}=0, a_3=a_4,$&\\
$ a_6=a_7, a_4 \neq 0$&$\langle \langle  X^2Y^2,{X^5} \rangle \rangle$\\
\hline
$a_1=a_{10}=1, a_2=a_5=a_8=a_9=a_{11}=a_4=0, a_3=a_4,$&\\
$a_6=a_7, a_{12} \neq 0$&$\langle \langle X^2Y^2,X^3Y,\bm{X^7} \rangle \rangle$\\
\hline
$a_1=a_{10}=1, a_2=a_5=a_8=a_9=a_{11}=a_4=a_{12}=0, a_3=a_4,$&\\
$a_6=a_7, a_{7} \neq 0$&$\langle \langle X^2Y^2,X^5 \rangle \rangle$\\
\hline
$a_1=a_{10}=1, a_2=a_5=a_8=a_9=a_{11}=a_4=a_{12}=a_7=0,$&\\
$ a_3=a_4,a_6=a_7$&$\langle \langle X^2Y,\bm{X^6} \rangle \rangle $\\
\hline
\hline
Intersection & $\langle \langle X^2Y^2,X^6Y\rangle \rangle  $\\
\hline
\end{tabular}
\end{center}
\end{table*}

\begin{table*}
\begin{center}
\caption[htbp]{$X^3Y^2+a_1X^6+a_2X^4Y+a_3X^2Y^2+a_4X^5+a_5X^3Y+$\\ $a_6XY^2+a_7X^4a_8X^2Y+a_9Y^2+a_{10}X^3+a_{11}XY+a_{12}X^2+a_{13}Y+a_{14}X+a_{15}$}
\label{tab8}
\begin{tabular}{|l|l|}
\hline
Coefficients&Subset of $\Box_{\prec_w}(F)$\\
\hline 
$a_1 \neq 1$&$\langle \langle X^3Y^2,X^6Y\rangle \rangle $\\
\hline
$a_1=1, a_2 \neq 0$&$\langle \langle X^3Y^2,X^7\rangle \rangle $\\
\hline
$a_1=1, a_2=0, a_3 \neq a_4$&$\langle \langle  X^3Y^2,X^5Y \rangle \rangle$\\
\hline
$a_1=1, a_2=0, a_3=a_4, a_5\neq0$&$\langle \langle X^3Y^2,X^6 \rangle \rangle $\\
\hline
$a_1=1, a_2=a_5=0, a_3=a_4, a_6\neq a_7$&$\langle \langle  X^3Y^2, X^4Y \rangle \rangle$\\
\hline
$a_1=1, a_2 =a_5=0, a_3=a_4, a_6=a_7, a_8\neq 0$&$\langle \langle X^2Y^2,\bm{X^6} \rangle \rangle$\\
\hline
$a_1=1, a_2=a_5=a_8=0, a_3=a_4, a_6=a_7, a_{9}\neq a_{10}$&$\langle \langle X^3Y, {\bm{X^7}} \rangle \rangle$\\
\hline
$a_1=1, a_2=a_5=a_8=0, a_3=a_4, a_9=a_{10}, a_{11} \neq 1$&$\langle \langle  XY^2 , {\bm{X^5Y}}, {\bm{X^6}}\rangle \rangle $\\
\hline
$a_1=1, a_2=a_5=a_8=a_{11}=0, a_3=a_4,  a_9=a_{10}$&$\langle \langle  X^3Y^2,X^4 \rangle \rangle$\\
\hline
\hline
Intersection & $\langle \langle  X^3Y^2\rangle \rangle $\\
\hline
\end{tabular}
\end{center}
\end{table*}

We conclude this section by collecting the established information on primary codes. First we introduce some notation for the general situation of primary affine variety codes. For each $M_i \in \Delta_{\prec}(I_q)$ let $\sigma(M_i)$ be the minimal number of monomials in $\Delta_\prec(I_q)$ which by some given method have been shown to be leading monomials of expressions of the form~(\ref{eqgrpol}), the minimum being taken over all polynomials $F$ having $M_i$ as leading monomial. Recall, that the relative distance between a pair of nested linear codes $C^{\prime \prime} \subsetneq C^{\prime }$ is given by
$$d(C^\prime, C^{\prime \prime} )=\min \{ w_H(\vec{c}) \mid \vec{c} \in C^{\prime} \backslash C^{\prime \prime} \}.$$
We then have the following theorem, the last part of which was not treated in~\cite{geilozbudak}, but which is included here due to its importance in connection with the CSS construction of asymmetric quantum codes (Section~\ref{sec4}).
\begin{theorem}\label{thebb}
The minimum distance of $C(I,L)$ is at least $$\min \{ \sigma(M) \mid M \in L\}.$$
Let  $L_2 \subsetneq L_1 \subseteq \Delta_\prec(I_q)$ and define 
\begin{equation}
n(L_1,L_2)=\min \{i \mid M_i \in L_1 \backslash L_2\}.\nonumber 
\end{equation} 
Then the relative distance $d(C (I, L_1),C(I,L_2) )$ is greater than or equal to
$$\min \{ \sigma(M_i) \mid M_i \in L_1, n(L_1,L_2) \leq i \}.$$  
\end{theorem}
\begin{proof}
To see the last part note that if $\vec{c} \in C(I,L_1) \backslash C(I,L_2)$ then $\vec{c}={\mbox{ev}}(F)$ for some $F$ with support in $L_1$ and with at least one monomial in the support not belonging to $L_2$. The last property implies that $M_{n(L_1,L_2)} \preceq {\mbox{lm}}(F)$. 
\end{proof}

For later reference, for $1 \leq s \leq n$ we define $E(s)={\mbox{Span}}_{\mathbb{F}_q}\{M_1, \ldots , M_s\}$ and $\tilde{E}(\delta) ={\mbox{Span}}_{\mathbb{F}_q}\{ M \mid \sigma(M) \geq \delta\}$. The latter code is said to be of designed minimum distance $\delta$ and as far as our analysis goes these codes have at least as good parameters as the first mentioned codes. 

Our treatment above of the Klein quartic immediately translates into the  estimates $\sigma(M_i)$, $i=1, \ldots , 22$ in Figure~\ref{figallprime}, from which it is straightforward to determine the dimension and to estimate the minimum distance of any corresponding code $E(s)$ and $\tilde{E}(\delta)$.
\begin{figure}
$$
\begin{array}{cccccccc}
M_{7}&M_{10}&M_{13}&M_{16}&M_{19}&M_{21}&M_{22}& \\
M_3&M_5&M_8&M_{11}&M_{14}&M_{17}&M_{20}&\\
M_1&M_2&M_4&M_6&M_9&M_{12}&M_{15}&M_{18}\\
\ \\
Y^2&XY^2&X^2Y^2&X^3Y^2&X^4Y^2&X^5Y^2&X^6Y^2\\
Y&XY&X^2Y&X^3Y&X^4Y&X^5Y&X^6Y\\
1&X&X^2&X^3&X^4&X^5&X^6&X^7\\
\ \\
13&10&7&5&3&2&1\\
18&15&12&9&6&4&2\\
22&19&16&13&10&7&4&1
\end{array}
$$
\caption{As a conclusion of Table~\ref{tab1} through Table~\ref{tab8} plus the observation prior to them, the figure contains in the lower part the estimates $\sigma(M)$ for all $M \in \Delta_{\prec_w}(I_8)$ }
 \label{figallprime}
\end{figure}

\section{From bounds on primary codes to bounds on dual codes}\label{sec2point5}

In this section we start by enhancing Theorem~\ref{thebb} to cover the general situation of primary linear codes. From that we then devise a result for general dual linear codes which is finally translated to the language of affine variety codes to obtain the counter part of Theorem~\ref{thebb} for dual affine variety codes. Theorem~\ref{the1} below in our opinion captures the very essence of Feng-Rao theory for primary codes, although to the best of our knowledge it has not been reported in this general version before. Recall that given vectors $\vec{a}=(a_1, \ldots , a_n)$ and $\vec{b}=(b_1, \ldots , b_n)$ the componentwise product is given by $\vec{a} \ast \vec{b}=(a_1b_1, \ldots , a_nb_n)$.

\begin{theorem}
\label{the1}
The following three statements are equivalent for a word $\vec{c} \in {\mathbb{F}}_q^n \backslash \{ \vec{0}\}$:
\begin{enumerate}
\item $w_H(\vec{c}) = w$ 
\item $w$ is the maximal integer for which there exists a vector space $V \subseteq {\mathbb{F}}_q^n$ of dimension $w$ such that for any $\vec{v} \in V \backslash \{ \vec{0}\}$ it holds that $\vec{c} \ast \vec{v} \neq \vec{0}$.
\item $w$ is the maximal integer for which there exists a linearly independent set\\
$\{ \vec{w}_1, \ldots , \vec{w}_w\} \subseteq {\mathbb{F}}_q^n$ and a corresponding set of vectors $\{ \vec{v}_1, \ldots , \vec{v}_w\} \subseteq {\mathbb{F}}_q^n$ such that $\vec{c}\ast \vec{v}_1=\vec{w}_1, \ldots , \vec{c} \ast \vec{v}_w=\vec{w}_w$.
\end{enumerate}
\end{theorem}
\begin{proof}\ \\
{{1. $\Leftrightarrow$ 2.:}} \ \ 
It is enough to show that $w_H(\vec{c})\geq w$ if and only if there exists a space $V$ of dimension at least $w$ satisfying the conditions in 2. 
We first observe that if $w_H(\vec{c}) \geq w$ then a space spanned by $w$ pairwise different standard vectors with a $1$ in a position where the corresponding entry of $\vec{c}$ is non-zero satisfies the conditions in 2.  Next, let $V$ be a vector space of dimension at least $w$ satisfying the conditions of 2. Aiming for a contradiction assume $w_H ( \vec{c}) < w$. But then there exist two different vectors in $V$ which have identical entries in those positions where $\vec{c}$ is non-zero, and therefore the componentwise product with $\vec{c}$ are the same. The difference between the vectors belongs to $V\backslash \{ \vec{0}\}$, but satisfies that the componentwise product with $\vec{c}$ equals $\vec{0}$ which is a contradiction.\\
{{2. $\Leftrightarrow$ 3.:}} \ \ 
It is enough to prove that the conditions of 2. are satisfied for a number $w$ if and only if the conditions of 3. are satisfied for the same number. First assume that the conditions of 2. hold for a given $w$.
The set $\{\vec{c} \ast \vec{v} \mid \vec{v} \in V\}$ is a vector space, but as by assumption $\vec{c} \ast \vec{v}^\prime \neq \vec{c} \ast \vec{v}^{\prime \prime}$ for different vectors $\vec{v}^\prime, \vec{v}^{\prime \prime} \in V$, this vector space is of dimension $w$ implying that the conditions of 3 hold. Next assume that the conditions of 3. hold for a given $w$.  We then have
$$\vec{c} \ast \sum_{i=1}^w a_i \vec{v}_i = \sum_{i=1}^w a_i \vec{c} \ast \vec{v}_i = \sum_{i=1}^w a_i \vec{w}_i \neq \vec{0},$$
when not all coefficients $a_1, \ldots , a_w$ equal $0$.
In particular $\{\vec{v}_1, \ldots , \vec{v}_w\}$ must be a linearly independent set as $\vec{c} \ast \vec{0}=\vec{0}$, implying that the conditions of 2. hold.
\end{proof}

\begin{remark}\label{remprimi}
Theorem~\ref{thebb} can be viewed as being a consequence of Theorem~\ref{the1}. To see this, note that if given $F(X_1, \ldots , X_m)$ with support in $\Delta_\prec(I_q)$ we have established polynomials $G_1(X_1, \ldots , X_m), \ldots , G_\sigma(X_1, \ldots , X_m)$ as in~(\ref{eqgrpol}) having pairwise different leading monomials all of which belong to $\Delta_\prec(I_q)$ then $\vec{w}_1={\mbox{ev}}(G_1), \ldots , \vec{w}_\sigma={\mbox{ev}}(G_\sigma)$ are linearly independent and for $\vec{c}={\mbox{ev}}(F)$ they satisfy the conditions of part 3 in Theorem~\ref{the1}. Actually, the implication also holds in the reverse direction, meaning that Theorem~\ref{the1} can be viewed as being a consequence of the results in the previous section. The reason for this is that equality holds in~(\ref{eqbox}).
\end{remark}
We next show that Theorem~\ref{the1} implies a bound for dual codes.
\begin{corollary}
\label{cor1}
Consider a basis $\{ \vec{b}_1, \ldots , \vec{b}_n \}$ for ${\mathbb{F}}_q^n$ as a vector space over ${\mathbb{F}}_q$. Given a non-zero word $\vec{c} \in {\mathbb{F}}_q^n$, let $m$ be the smallest index such that $\vec{c} \cdot \vec{b}_m \neq 0$. If there exists a vector space $V \subseteq {\mathbb{F}}_q^n$ of dimension $w$ such that for any $\vec{v} \in V \backslash \{0\}$ a word $ \vec{u} \in {\mathbb{F}}_q^n$ exists  with $$\vec{v} \ast \vec{u} \in {\mbox{Span}}_{\mathbb{F}_q}\{\vec{b}_1, \ldots , \vec{b}_m\} \backslash {\mbox{Span}}_{\mathbb{F}_q} \{ \vec{b}_1, \ldots , \vec{b}_{m-1}\}$$
then $w_H(\vec{c}) \geq w$.
\end{corollary}
\begin{proof}
We have $\vec{c}\cdot (\vec{v} \ast \vec{u}) \neq 0$ which can be rewritten $(\vec{c} \ast \vec{v}) \cdot \vec{u} \neq 0$. But then $\vec{c}\ast \vec{v} \neq \vec{0}$ and the result follows from Theorem~\ref{the1}. 
\end{proof}

\begin{remark}
Corollary~\ref{cor1} can be viewed as being a generalization of all previous versions of the Feng-Rao bound for dual codes (e.g.~\cite{FR24,FR1,FR2,MM,miura1998linear,MR1268660,agismm,salazar}), except~\cite{geilmartin2013further}[Th.\ 19] which has another flavor. It also has a strong relation to~\cite{MR2332494}.
\end{remark}
Combining Corollary~\ref{cor1} and the arguments in Remark~\ref{remprimi} we obtain a bound on the minimum distance as well as relative distance of  dual affine variety codes having the same flavor as Theorem~\ref{thebb}.
Write as in Section~\ref{sec2} 
$$\Delta_\prec (I_q)=\{ M_1, \ldots, M_n\}$$
where the enumeration is according to the fixed monomial ordering $\prec$. For $i=1, \ldots , n$ define
\begin{eqnarray}
\Lambda(M_i)&=&\{ M \in \Delta_\prec(I_q) | {\mbox{ for each }} F {\mbox{ with }} {\mbox{lm}}(F)=M 
{\mbox{ and }} {\mbox{Supp}}(F) \subseteq \Delta_\prec(I_q) \nonumber \\
&& {\mbox{ \ \ \ there exists an expression of the form (\ref{eqgrpol}) with leading monomial }} \nonumber \\
&&{\mbox{ \ {\hspace{9.5cm}}  equal to }} M_i\} \nonumber
\end{eqnarray}
Here, ${\mbox{Supp}}$ means the support. Let for $i=1, \ldots, n$, $\mu(M_i)$ be a lower bound estimate on $\# \Lambda(M_i)$. We then have the following result.

\begin{theorem}\label{thebbb}
Let $L \subseteq \Delta_\prec(I_q)$. Then the minimum distance of $C^\perp(I,L)$ is at least $$\min \{ \mu(M) \mid M \notin L\}.$$
Consider  $L_2 \subsetneq L_1\subseteq \Delta_\prec(I_q)$ and define 
\begin{equation}
m(L_1)=\max \{ i \mid M_i \in L_1\}.\nonumber 
\end{equation} 
The relative distance $d(C^\perp (I, L_2),C^\perp (I,L_1)) $ is greater than or equal to
$$\min \{ \mu(M_i) \mid M_i \notin L_2, i \leq m(L_1)\}.$$  
\end{theorem}
\begin{proof}
The first part follows from Corollary~\ref{cor1} and similar arguments as in Remark~\ref{remprimi}. Let $\vec{c}\in C^\perp(I,L_2)$ but $\vec{c} \notin C^\perp (I,L_1)$, and define as in Corollary~\ref{cor1} $m$ to be the smallest index such that $\vec{c} \cdot {\mbox{ev}}(M_m) \neq 0$. From the first property it follows that $M_m \notin L_2$, and from the latter that $m \leq m(L_1)$.
\end{proof}
We shall apply two code constructions, namely $\tilde{C}(\delta)=C^\perp (I,L)$ with
\begin{equation*}
L=\{M \mid \mu (M) < \delta \}
\end{equation*}
and $C(s)=C^\perp(I,L)$ with 
\begin{equation*}
L=\{M_1, \ldots , M_s\}.
\end{equation*}
The first code is so to speak of designed minimum distance $\delta$ and clearly it is of dimension at least that of any code $C(s)$ for which we know that the minimum distance is at least $\delta$. 
\begin{remark}
Writing $C(0)={\mathbb{F}}_q^n$, by definition $\mu(M_i)$ is a lower bound estimate on the minimal Hamming weight of a code word in $C(i-1) \backslash C(i)$, or in other words on $d(C(i-1),C(i))$.
\end{remark}
We conclude this section by observing that in the spirit of~\cite{salazar} it is possible to formulate a bound that can potentially be even sharper than Theorem~\ref{thebbb}. For the Klein quartic and the given $\prec_w$ this enhancement does not produce better results and we therefore only give a very brief description.  For each $S \subseteq \Delta_\prec(I_q)$ define

\begin{eqnarray} 
\Lambda(M_i,S)&=&\{ M \in S | {\mbox{ for each }} F {\mbox{ with }} {\mbox{lm}}(F)=M 
{\mbox{ and }} {\mbox{Supp}}(F) \subseteq S \nonumber \\
&& {\mbox{ \ \ there exists an expression of the form (\ref{eqgrpol}) with leading monomial }} \nonumber \\
&&{\mbox{ \ {\hspace{9.5cm}} equal to }} M_i\}. \nonumber 
\end{eqnarray}

The idea then is that for each $i$ one can choose $S=S_i$ in such a way that $\Lambda(M_i,S_i)$ becomes as large as possible and possibly larger than $\Lambda(M_i)$, the bound in Theorem~\ref{thebbb} being then possibly improved by replacing the latter with the former.

\section{Dual affine variety codes from the Klein quartic}\label{sec3}
We now apply the method from the previous section on dual affine variety codes to the special case of the Klein quartic. Throughout this section therefore $I =\langle Y^3+X^3Y+X\rangle  \subseteq {\mathbb{F}}_8[X,Y]$ and $\prec$ equals $\prec_w$ as defined in Section~\ref{sec2}. To apply Theorem~\ref{thebbb} we collect information on $\Lambda(M_i)$ as follows. Firstly, for any $X^iY^j \in \Delta_{\prec_w}(I_8)$ we have 
\begin{eqnarray}
\{ X^{i^\prime}Y^{j^\prime} \mid 0\leq i^\prime \leq i, 0 \leq j^\prime \leq j\} \subseteq \Lambda(X^iY^j).\label{eqtrivial}
\end{eqnarray}
Additional information can be derived by carefully inspecting the last row in Table~\ref{tab1} through Table~\ref{tab7} (Table~\ref{tab8} turns out to contribute with no extra information as the last row in it consists of nothing but the monomials divisible by the $M_i$ under consideration). 
\begin{example}
In this example we derive information on $\Lambda(X^6)$. From~(\ref{eqtrivial}) we see that $\{ 1, X,X^2,X^3,X^4,X^5,X^6\}$ is contained herein. Additional elements are found by inspecting for which of Tables~\ref{tab1} through \ref{tab7} we have that $X^6$ belongs to the intersection specified in the last row. This happens for Tables~\ref{tab1} through \ref{tab5} and therefore also $Y, XY, X^2Y, Y^2, XY^2$ belong to $\Lambda(X^6)$.
\end{example}
In Section~\ref{sec4.5} we shall demonstrate that the size of the subsets of $\Lambda(M_i)$ as depicted in Table~\ref{tab9} in a considerable amount of cases gives the true value of $d(C(i-1),C(i))$. However, it is also shown in Example~\ref{exfino} of that section that $$\mu(M_{18}=X^7)=\# (\{1, X, \ldots , X^7\}\cup \{ Y, XY, X^2Y, Y^2, XY^2, X^3Y\})=14$$ is not  the true value of $d(C(17),C(18))$ as suggested by Table~\ref{tab9}, but that this actually equals $15$. 
We collect our findings, including the result of Example~\ref{exfino}, in Table~\ref{tab9} from which we obtain the lower bound estimates on $d(C(i-1),C(i))$ as presented in Figure~\ref{figalldual}.
\begin{table}
\caption{Information on $\Lambda(M)$ for the Klein quartic and $\prec_w$}
\label{tab9}
\begin{center}
\begin{tabular}{|c|l|}
\hline
$M$ & Established subset of $\Lambda (M)$\\
\hline
\hline
$X^4$&$\{1,X,X^2,X^3,X^4\} \cup \{Y\}$\\
\hline
$X^5$&$\{1, X, X^2,X^3,X^4,X^5\}\cup \{Y, XY\}$\\
\hline
$X^6$&$\{1, X, X^2,X^3,X^4,X^5,X^6\} \cup \{Y, XY, X^2Y, Y^2, XY^2\}$\\
\hline
$X^7$&$\{1, X, X^2,X^3,X^4,X^5,X^6,X^7\} \cup \{Y, XY, X^2 Y, Y^2, XY^2, X^3Y\}$\\
 \hline
$X^4Y$&$\{1, X, X^2, X^3, X^4, Y, XY, X^2Y, X^3Y, X^4Y\} \cup \{Y^2\}$\\
\hline
$X^5Y$& $ \{1, X, X^2,X^3,X^4, X^5,Y,XY,X^2Y,X^3Y,X^4Y,X^5Y\} \cup \{ Y^2, XY^2\}$\\
\hline
$X^6Y$ & $ \{ 1, X, X^2,X^3,X^4,X^5,X^6,Y,XY,X^2Y,X^3Y,X^4Y,X^5Y,X^6Y\} $\\
&$\cup \{ Y^2,XY^2,X^2Y^2\}$\\
\hline
\end{tabular}
\end{center}
\end{table}
\begin{figure}
$$
\begin{array}{cccccccc}
M_{7}&M_{10}&M_{13}&M_{16}&M_{19}&M_{21}&M_{22}& \\
M_3&M_5&M_8&M_{11}&M_{14}&M_{17}&M_{20}&\\
M_1&M_2&M_4&M_6&M_9&M_{12}&M_{15}&M_{18}\\
\ \\
Y^2&XY^2&X^2Y^2&X^3Y^2&X^4Y^2&X^5Y^2&X^6Y^2\\
Y&XY&X^2Y&X^3Y&X^4Y&X^5Y&X^6Y\\
1&X&X^2&X^3&X^4&X^5&X^6&X^7\\
\ \\
3&6&9&12&15&18&21\\
2&4&6&8&11&14&17\\
1&2&3&4&6&8&12&15\\
\end{array}
$$
\caption{The monomials $M_i$ and $\mu(M_i)$, where for $i=18$ we redefined $\mu$ to be the improved estimate on $d(C(i-1),C(i))$ from Example~\ref{exfino} in Section~\ref{sec4.5}.}
\label{figalldual}
\end{figure}
Applying the information in the lower part of Figure~\ref{figalldual} we establish code parameters of $\tilde{C}(\delta)$ as in Table~\ref{tabcodepardual}.
\begin{table}
\caption{Estimated parameters of the codes $\tilde{C}(\delta)$ as found in the present paper}
\label{tabcodepardual}
\begin{center}
\begin{tabular}{ccc}
$[22,21,2]_8$&$[22, 19,3]_8$&$[22, 17,4]_8$\\
$[22,15,6]_8$&$[22, 12,8]_8$&$[22, 10,9]_8$\\
$[22,9,11]_8$&$[22, 8,12]_8$&$[22, 6,14]_8$\\
$[22,5,15]_8$&$[22, 3,17]_8$&$[22, 2,18]_8$\\
$[22,1,21]_8$
\end{tabular}
\end{center}
\end{table}
For $9$ out of the $13$ codes mentioned in Table~\ref{tabcodepardual}, for the given dimension the designed minimum distance equals the best value known to exist according to~\cite{grassl}. For the remaining dimensions ($10$, $3$, $2$ and $1$) the value in \cite{grassl} exceeds ours by one.

Regarding~\cite{kolluru-feng-rao}[ Ex.\ 3.2, Ex.\ 4.1] a comparison between their results and ours can be taken as a direct measure of how well their method and our method compete. In Table~\ref{tabkfr} we list parameters as can be derived using their findings. Here, entries in bold come from \cite{kolluru-feng-rao}[Ex.\ 4.1] whereas the remaining cases come from \cite{kolluru-feng-rao}[Ex.\ 3.2]. It is evident that we always find at least as sharp estimates than they do, and that we do significantly better in a considerable amount of cases. For instance we have $[22,9,11]_8$ and $[22,8,12]_8$ codes whereas they produce $[22,9,9]_8$ and $[22,7,10]_8$. Similarly our $[22,5,15]_8$ compares favorable with their $[22,3,15]_8$. In Table~\ref{tabkfr} we do  not include the parameters $[22,16,5]_8$ which in \cite{kolluru-feng-rao}[Tab.\ 2] is claimed to be demonstrated in their Example 4.1. However, there is no treatment of $d=5$ in that example, and it does not seem possible to produce such parameters using their method.
\begin{table}
\caption{Estimated parameters as found in~\cite{kolluru-feng-rao}[Ex.\ 3.2, Ex.\ 4.1]}
\label{tabkfr}
\begin{center}
\begin{tabular}{ccc}
$[22,21,2]_8$&$[22,19,3]_8$&$[22,17,4]_8$\\
$[22,15,5]_8$&${\bm{[22,15,6]_8}}$&$[22,11,7]_8$\\
$[22,10,8]_8$&${\bm{[22,9,9]_8}}$&$[22,7,10]_8$\\
$[22,6,12]_8$&$[22,4,14]_8$&$[22,3,15]_8$\\
$[22,2,18]_8$&$[22,1,21]_8$
\end{tabular}
\end{center}
\end{table}

\section{From parity-check matrix to generator matrix (and vice versa)}\label{sec5}

Having in the previous sections established results on primary affine variety codes and dual affine variety codes it is a natural question to ask how these two constructions relate to each other. In the present section we answer this question by establishing affine variety descriptions of generator matrices for the dual affine variety codes. Of course this immediately translates to a result in the reverse direction producing similar descriptions of parity-check matrices for primary affine variety codes. 

In full generality, given an affine variety ${\mathbb{V}}(I_q)$ of size $n$ let $\{ F_1+I_q, \ldots , F_n+I_q\}$ be a basis for ${\mathbb{F}}_q[X_1, \ldots , X_m]/I_q$, where we may assume that the support of each of the polynomials $F_i$ is a subset of $\Delta_\prec(I_q)$. We describe a procedure to determine another basis for ${\mathbb{F}}_q[X_1, \ldots , X_m]/I_q$, with the representatives again having support in $\Delta_\prec(I_q)$, from which one can immediately for arbitrary 
\begin{equation}
\big( {\mbox{Span}}_{\mathbb{F}_q}\{{\mbox{ev}}(F_{i_1}) ,\ldots , {\mbox{ev}}(F_{i_{n-k}})\}\big)^\perp \label{eqgendual}
\end{equation}
 read of a primary description.
The method 
involves simple Gaussian elimination and Lagrange interpolation adapted to the affine variety ${\mathbb{V}}(I_q)$. It is universal in the sense that it simultaneously solves the problem for all choice of index $\{i_1, \ldots , i_{n-k}\} \subseteq \{1, \ldots , n\}$. In particular the method simultaneously returns affine variety descriptions of a generator matrix for $C^\perp(I,L)$ for all choices of $L \subseteq \Delta_\prec(I_q)$, which of course then implies affine variety descriptions of a parity-check matrix for all $C(I,L)$). This is in contrast to the method in~\cite{MR4272610} where calculations have to be repeated for every particular affine variety code.

We start by observing that given a basis ${\mathcal{B}}=\{\vec{b}_1, \ldots , \vec{b}_n\}$ for ${\mathbb{F}}_q^n$ as a vector space over ${\mathbb{F}}_q$ we can find another basis ${\mathcal{B}}^\perp=\{ \vec{b}_1^\perp, \ldots , \vec{b}_n^\perp \}$ such that for any set of $n-k$ pairwise different elements $\{i_1, \ldots , i_{n-k}\} \subseteq \{1, \ldots , n\}$ it holds that
\begin{equation}
\big({\mbox{Span}}_{\mathbb{F}_q}\{\vec{b}_{i_1}, \ldots , \vec{b}_{i_{n-k}}\} \big)^\perp={\mbox{Span}}_{\mathbb{F}_q}\{\vec{b}_j^\perp \mid j \in \{1, \ldots , n\} \backslash \{i_1, \ldots , i_{n-k}\}\} \label{eqsnabel1}
\end{equation}
(and vice versa). The proof of the following theorem establishing such a basis is straightforward.

\begin{theorem}\label{thefromgentopar}
Given a basis ${\mathcal{B}}$ of row vectors as above write
$$B=\left[ \begin{array}{c}
\vec{b}_1 \\
\vec{b}_2 \\
\vdots \\
\vec{b}_n
\end{array}
\right]
$$
then
$$B^{-1}=\left[ \big( \vec{b}_1^\perp \big)^T,  \big( \vec{b}_2^\perp \big)^T,\ldots , \big( \vec{b}_n^\perp \big)^T \right] $$
where  $\vec{b}_j^\perp$, $j=1, \ldots , n$ are as in~(\ref{eqsnabel1}).
\end{theorem}
\begin{proof}
See~\cite{agismm}.
\end{proof}

We next describe Lagrange interpolation over any affine variety $$V={\mathbb{V}}(I_q)=\{P_1, \ldots , P_n\} \subseteq {\mathbb{F}}_q^m.$$ Recall, that given a point $(a_1, \ldots , a_m) \in {\mathbb{F}}_q^m$ the Lagrange  basis polynomial
\begin{equation}
L_{(a_1,\ldots , a_m)}(X_1, \ldots , X_m)=\prod_{i=1}^m \prod_{s \in {\mathbb{F}}_q\backslash \{a_i\}} \frac{X_i-s}{a_i-s} \label{eqlagrange1}
\end{equation}
evaluates to $1$ in $(a_1, \ldots , a_m)$ and to $0$ in all other points of ${\mathbb{F}}_q^m$. Before continuing, for each $(a_1, \ldots ,  a_m) \in V$ we calculate the remainder of $L_{(a_1, \ldots  ,a_m)}$ modulo a Gr\"{o}bner basis for $I_q$ the result of which we denote $L_{(a_1, \ldots , a_m)}^{(V)}$. This new polynomial evaluates in the same way for all points in $V$, but possibly differently outside. Observe, that the support of the latter polynomial clearly is contained in $\Delta_\prec(I_q)$. Given $(c_1, \ldots , c_n)$ we then obtain the polynomial $G$ with support in $\Delta_\prec(I_q)$ and satisfying\footnote{We remark that in cases where for a given index $i$ there exists an $s \in {\mathbb{F}}_q$ such that for no $P_j$ the $i$th coordinate is $s$ we may leave out this value in the $i$th product of~(\ref{eqlagrange1}) simplifying the calculations. }
\begin{equation*}
G(P_i)=c_i, i=1, \ldots, n 
\end{equation*}
as follows
$$G=\sum_{i=1}^n c_iL_{P_i}^{(V)}.$$

We now combine the two mentioned results to solve the problem posed at the beginning of the section. 
\begin{procedure}
Consider a set of linearly independent polynomials $F_1, \ldots , F_{n=\#\Delta_\prec(I_q)}$ over $\mathbb{F}_q$ all with support in $\Delta_\prec(I_q)$. E.g.\ $\{F_1=M_1, \ldots , F_n=M_n\} = \Delta_\prec(I_q)$  where $M_i \prec M_j$ for $i<j$. We first calculate $\vec{b}_i={\mbox{ev}}(F_i)$ for $i=1, \ldots , n$. Then using Theorem~\ref{thefromgentopar} we determine $\vec{b}_1^\perp, \ldots , \vec{b}_n^\perp$ satisfying~(\ref{eqsnabel1}). Finally, for $i=1, \ldots , n$ we apply Lagrange interpolation over the variety ${\mathbb{V}}(I_q)$ to determine polynomials $F_i^\perp$ with support in $\Delta_\prec(I_q)$ such that ${\mbox{ev}}(F_i^\perp)=\vec{b}_i^\perp$. Then~(\ref{eqgendual}) equals ${\mbox{Span}}_{\mathbb{F}_q}\{ {\mbox{ev}}(F_j^\perp) \mid j \in \{ 1, \ldots , n\} \backslash \{i_1, \ldots , i_{n-k}\}\}$. In particular for any $L \subseteq \Delta_\prec(I_q)$ we have
$$C^\perp(I,L) = {\mbox{Span}}_{\mathbb{F}_q}\{ {\mbox{ev}}(F_i^\perp) \mid M_i \notin L\}$$
when $F_1=M_1, \ldots , F_n=M_n$.
\end{procedure}

For the Klein quartic the variety consists of the points in Table~\ref{tabpoints}. This information was used as input when implementing the procedure on a computer system to obtain the generating polynomials $F_1^\perp , \ldots , F_{22}^\perp$ of the dual basis as described in Table~\ref{tabgenone} for the case of the Klein quartic and the monomial ordering being $\prec_w$. It is a manageable task to check that indeed $\sum_{s=1}^{22}M_i (P_s) F_j^\perp(P_s)=\delta_{i,j}$, where $\delta_{i,j}$ denotes the Kronecker function. From Table~\ref{tabgenone} it is straight forward to devise generator matrices for the dual affine variety codes of Section~\ref{sec3} and similarly parity-check matrices for the primary affine variety codes in~\cite{geilozbudak}.
\begin{table}
\caption{The points of the Klein quartic. Here, $\alpha$ is a root of $T^3+T^2+1$.}
\label{tabpoints}
\begin{tabular}{lll}
$P_1=( 0,0 )$&$P_2=( \alpha,\alpha )$&$P_3=(\alpha ,\alpha^2+1 )$\\
$P_4=(\alpha , \alpha^2+\alpha+1)$&
$P_5=( \alpha^2,\alpha )$&$P_6=(\alpha^2, \alpha^2 )$\\
$P_7=(\alpha^2 ,\alpha^2+\alpha )$&$P_8=(\alpha^2+1 ,\alpha^2+\alpha+1 )$&
$P_{9}=(\alpha^2+1 ,\alpha^2+\alpha )$\\
$P_{10}=(\alpha^2+1 ,1 )$&$P_{11}=(\alpha^2+\alpha+1 ,\alpha^2 )$&$P_{12}=(\alpha^2+\alpha+1 ,\alpha^2+\alpha+1 )$\\
$P_{13}=(\alpha^2+\alpha+1 , \alpha+1)$&$P_{14}=( \alpha+1, \alpha^2)$&$P_{15}=(\alpha+1 ,\alpha^2+1 )$\\
$P_{16}=(\alpha+1 ,1 )$&
$P_{17}=( \alpha^2+\alpha, \alpha)$&$P_{18}=(\alpha^2 +\alpha,\alpha+1 )$\\$P_{19}=(\alpha^2+\alpha ,1 )$&$P_{20}=(1 , \alpha^2+1)$&
$P_{21}=(1 ,\alpha+1 )$\\$P_{22}=(1 , \alpha^2+\alpha)$
\end{tabular}
\end{table}
\begin{table}
\caption{The output from the procedure for the Klein quartic and $\prec_w$}
\label{tabgenone}
\begin{tabular}{llllll}
\hline $M_1=1 $&$M_2=X $&$M_3= Y$&$M_4=X^2 $&$M_5=XY $&$M_6=X^3 $\\
$M_7=Y^2 $&$M_8=X^2Y $&$M_9=X^4 $&$M_{10}=XY^2 $&$M_{11}=X^3Y $&$M_{12}=X^5 $\\
$M_{13}=X^2Y^2 $&$M_{14}=X^4Y $&$M_{15}= X^6$&$M_{16}=X^3Y^2 $&$M_{17}=X^5Y $&$M_{18}=X^7 $\\
$M_{19}= X^4Y^2$&$M_{20}=X^6Y $&$M_{21}=X^5Y^2 $&$M_{22}=X^6Y^2 $\\
\hline
$F_1^\perp=X^7+1 $&$F_2^\perp=X^6 $&$F_3^\perp= X^6Y^2$&$F_4^\perp=X^5 $&$F_5^\perp= X^5Y^2$&$F_6^\perp=X^4 $\\
$F_7^\perp= X^6Y$&$F_8^\perp= X^4Y^2$&$F_9^\perp=X^3 $&$F_{10}^\perp=X^5Y $&$F_{11}^\perp=X^3Y^2 $&$F_{12}^\perp=X^2 $\\
$F_{13}^\perp=X^4Y $&$F_{14}^\perp=X^2Y^2 $&$F_{15}^\perp= X$&$F_{16}^\perp=X^3Y $&$F_{17}^\perp=XY^2 $&$F_{18}^\perp=1 $\\
$F_{19}^\perp= X^2Y$&$F_{20}^\perp=Y^2 $&$F_{21}^\perp=XY $&$F_{22}^\perp=Y $\\
\hline
\end{tabular}
\end{table}
\section{Tightness of estimates for dual codes}\label{sec4.5}
The information from the previous section allows us to prove tightness of the bound from Section~\ref{sec3} in a substantial number of cases and to improve upon the estimate $\mu(M_{18})=14$ from Section~\ref{sec3} on the minimal weight of words in $C(17)\backslash C(18)$.
\begin{example}\label{exforimpure}
By definition, $\mu(M_i)$ is a lower bound on $d(C(i-1),C(i))$. We now show that $\mu (M_7)=3$ constitutes a sharp bound. From Table~\ref{tabgenone} we see that $F_7^\perp=X^6Y$ and that among $F_8^\perp, \ldots , F_{22}^\perp$ we have $X^5Y, X^4Y, X^3Y, X^2Y, XY$. Therefore for arbitrary  $b \in {\mathbb{F}}_8^\ast$ we get
\begin{equation}
\vec{c} ={\mbox{ev}}\big(Y \! \!   \prod_{a \in {\mathbb{F}}_8^\ast\backslash \{ b\}}(X-a) \big) \in C(6)\backslash C(7). \label{eqhalloejsa}
\end{equation}
Inspecting Table~\ref{tabpoints} we find that among the points $P_1, \ldots, P_{22}$ there are exactly three having $b$ as the first coordinate (the second coordinate of these three points as expected being different from $0$). Therefore $w_H(\vec{c})=3$, and $\mu(M_7)=3$ is sharp.
From the $\mu$-sequence in Figure~\ref{figalldual} we conclude that $d(C^\perp(I,L))=3$ whenever $\{M_1, M_2, M_3\} \subseteq L\subseteq \{M_1, \ldots , M_6\}$, and in particular that $d(C(3))=d(C(4))=d(C(5))=d(C(6))=3$.
\end{example}
We shall return to Example~\ref{exforimpure} in the next section where it shall support us in constructing so-called impure asymmetric quantum codes.
\begin{example}
Using exactly the same type of arguments as in Example~\ref{exforimpure}, but having a decreasing number of terms $(X-a)$ in~(\ref{eqhalloejsa}) one can show that $\mu(M_{10})$, $\mu(M_{13})$, $\mu(M_{16})$, $\mu(M_{19})$, $\mu(M_{21})$ and $\mu(M_{22})$ are all sharp estimates. Inspecting the $\mu$-sequence we conclude:  {\it{(i)}} for $\{M_1, \ldots , M_7\} \subseteq L \subseteq \{M_1, \ldots , M_9\}$ it holds that $d(C^\perp(I,L))=\mu(M_{10})=6$ and in particular that $d(C(7))=d(C(8))=d(C(9))=6$, {\it{(ii)}}
 $d(C(12))=\mu(M_{13})=9$, {\it{(iii)}}
$d(C(14))=d(C(15))=\mu(M_{16})=12$, {\it{(iv)}}
$d(C(18))=\mu(M_{19})=15$, {\it{(v)}}
$d(C(20))=\mu(M_{21})=18$, and finally {\it{(vi)}} $d(C(21))=\mu(M_{22})=21$.
\end{example}
\begin{example}
In this example we show that $\mu(M_{11})$, $\mu(M_{14})$, $\mu(M_{17})$ and $\mu(M_{20})$ are all sharp. Write $\{ a_1, a_2,a_3\}=\{\alpha,\alpha^2+\alpha,\alpha^2+\alpha +1\}$ where $\alpha$ is a root of $T^3+T^2+1$. The polynomial
\begin{equation}
F=Y(X+Y+\alpha)\prod_{i=1}^3(X-a_i) \label{eqhiphurrah}
\end{equation}
has $14$ roots in ${\mathbb{V}}(I_8)$, namely from the first factor $P_1$, from the second factor $P_7$, $P_8$, $P_{16}$, $P_{21}$, and from the last factor $3$ times $3$ equals $9$ points none of which are the same as the previous mentioned, due to the way $\{a_1, a_2, a_3\}$ has been chosen. Hence, $w_H(\vec{c}={\mbox{ev}}(F))=22-14=8$. Observe that $\mu(M_{11})$ also equals $8$ and from Table~\ref{tabgenone} that $F_{11}^\perp \in {\mbox{Supp}}(F)$ and that all other monomials in the support belong to $\{F_{12}^\perp, \ldots , F_{22}^\perp\}$. Therefore $\vec{c}\in C(10) \backslash C(11)$ which means that $\mu(M_{11})$ is sharp. By replacing the three linear factors in the last factor of~(\ref{eqhiphurrah}) with only $2$, $1$ and $0$, respectively, we can show that also $\mu(M_{14})$, $\mu(M_{17})$ and $\mu(M_{20})$, respectively, are sharp. In particular $d(C(10))=8$, $d(C(13))=11$, $d(C(16))=14$, and $d(C(19))=17$.
\end{example}
\begin{example}
In this example we show that $\mu(M_{15})=12$ is a sharp estimate. The polynomial
\begin{eqnarray}
F&=&\alpha X^3Y+\alpha XY^2+(\alpha^2+\alpha+1)Y^2+(\alpha^2+\alpha +1)XY+(\alpha +1)Y+X \nonumber
\end{eqnarray}
has $F_{15}^\perp$ in its support and all other monomials in the support belong to $\{F_{16}^\perp , \ldots , F_{22}^\perp\}$ which implies $\vec{c}={\mbox{ev}}(F) \in C(14) \backslash C(15)$. By inspection $F$ has $22-10=12$ non-roots in ${\mathbb{V}}(I_8)$.
\end{example}
\begin{example}\label{exfino}
From Table~\ref{tab9} in Section~\ref{sec3} we have the estimate $\mu(M_{18})=14$ for the minimal Hamming weight of a word in $C(17)\backslash C(18)$. However, the true minimal value equals $15$ as we now demonstrate. From Table~\ref{tabgenone} we know that a word as above can be written $\vec{c}={\mbox{ev}}(F)$ where 
$$1 \in {\mbox{Supp}}(F) \subseteq \{1, X^2Y, Y^2,XY,Y\}$$
hence, without loss of generality we write $F=f_1X^2Y+f_2Y^2+f_3XY+f_4Y+1$. Considering first the case $f_1=0$, but $f_2 \neq 0$, from line 4 of Table~\ref{tab3} we see $w_H(\vec{c}) \geq \# \langle \langle Y^2, X^3 \rangle \rangle =16$. Next we consider $f_1=f_2=0$, but $f_3 \neq 0$. Then from line 4 of Table~\ref{tab2} we have $w_H(\vec{c}) \geq \# \langle \langle XY, Y^2, X^5\rangle \rangle =16$. Now consider $f_1=f_2=f_3=0$, but $f_4 \neq 0$. In this case the second line of Table~\ref{tab1} tells us $w_H(\vec{c})\geq \# \langle \langle Y, X^3 \rangle \rangle =19$. For the case $f_1=f_2=f_3=f_4=0$, obviously $w_H(\vec{c})=22$. What remains is to consider the case $f_1 \neq 0$ which we now embarg on. From Table~\ref{tab4} we do have some information on $\Box_{\prec_w}(F)$ in that case, but we need additional analysis to arrive at our conclusion. Scaling $F$ by a non-zero constant we shall in the following rewrite it as
$F=X^2Y+bY^2+cXY+dY+a$, $a\neq 0$. Applying the method described in Section~\ref{sec2} we now seek for an expression of the form~(\ref{eqgrpol}) that will allow us to add an extra monomial to $\Box_{\prec_w}(F)$. We have
$$Y^2F=X^2Y^3+bY^4+cXY^3+dY^3+aY^2$$ 
from which we subtract 
$(X^2+bY+cX+d)(Y^3+X^3Y+X)$ to obtain
$$X^5Y+X^3+bX^3Y^2+bXY+cX^4Y+cX^2+dX^3Y+dX+aY^2.$$
From this we subtract $X^3F$ and end up with
$$(a+1)X^3+bXY+cX^2+dX+aY^2$$
which has $Y^2$ as leading monomial. Combining this information with line 4 of Table~\ref{tab4} we see that 
$$\langle \langle Y^2,X^2Y,X^5 \rangle \rangle \subseteq \Box_{\prec_w}(F)$$
holds, the set on the left-hand-side being of size $15$. We have considered all possible polynomials $F$ such that $\vec{c}={\mbox{ev}}(F) \in C(17)\backslash C(18)$ and can conclude $w_H(\vec{c})\geq 15$. Finally we show sharpness of our bound. Consider namely $F=\alpha X^2Y+1$. Then $F$ is of the prescribed form and by inspection it has exactly $22-7=15$ non-roots in ${\mathbb{V}}(I_8)$.
\end{example}

\section{Asymmetric quantum codes from the Klein quartic}\label{sec4}
Having in the previous sections treated classical linear codes, we in this section give some examples of good asymmetric quantum codes derived from nested codes over the Klein quartic under $\prec_w$. We start with a brief introduction to the concept.

A linear $q$-ary asymmetric quantum code is a $q^k$-dimensional subspace of the Hilbert space ${\mathbb{C}}^{q^n}$ which enables that $d_Z-1$ phase-shift errors as well as  $d_X-1$ qudit-flip errors can be detected and consequently that $\lfloor \frac{d_Z-1}{2}\rfloor$ phase-shift errors as well as  $\lfloor \frac{d_X-1}{2} \rfloor$ qudit-flip errors can be corrected~\cite{aq6}. The usual notation for this is $[[n,k,d_Z/d_X]]_q$. An important way of constructing such a code is through the use of the CSS-construction~\cite{aq7,aq6} which we recall in the theorem below.  In the following, to avoid confusion with the dimensions of the involved classical codes we use the symbol $\ell$ for the dimension of the asymmetric quantum code.

\begin{theorem}\label{thecss}
Let $C_2 \subsetneq C_1 \subseteq {\mathbb{F}}_q^n$ be linear codes with $\ell=\dim C_1 - \dim C_2$. Applying the CSS construction one obtains an asymmetric quantum code with parameters 
$[[n,\ell, d_Z/d_X]]_q$ where $d_Z=d(C_1,C_2)\geq d(C_1)$ and $d_X=d(C_2^\perp, C_1^\perp)\geq d(C_2^\perp)$.
\end{theorem}

An asymmetric quantum code from the CSS construction is called pure when both $d(C_1,C_2)=d(C_1)$ and $d(C_2^\perp , C_1^\perp)=d(C_2^\perp)$ holds true and otherwise impure or degenerate. Impure codes can be desirable due to their ability to support passive error-correction~\cite{4036136,ezerman} and in connection with hybrid codes for simultaneous transmission of quantum and classical information~\cite{8006823}, but it is considered to be a difficult research problem to establish examples of them. The machinery of the present paper, however, allows us to present at the end of this section a couple of new examples. 

Observe that in Theorem~\ref{thecss} the existence of an $[[n,\ell,d^{\prime}/d^{\prime \prime}]]_q$ code is equivalent to the existence of an $[[n,\ell, d^{\prime \prime}/d^\prime]]_q$ code, which is seen by replacing the nested codes with their duals. Therefore if for instance we want $d_Z \geq d_X$ (which can be desirable~\cite{aq7,aq6}) we may consider a case where the relative distance between two primary codes is strictly smaller than the relative distance between their duals as long as we apply the mentioned correspondence at the end. In particular if we from two sets of nested codes conclude the existence of quantum codes with parameters $[[n,\ell,a/b]]_q$ and $[[n,\ell,c/a]]_q$, respectively, then if $c>b$ we shall conclude the existence of a code with parameters $[[n,\ell ,a/c]]_q$ and similarly the existence of a code with parameters $[[n,\ell,b/a]]_q$ when $b >c$.  In the remaining part of this paper when writing $[[n,\ell,d_Z/d_X]]_q$ we mean a code of length $n$, dimension $\ell$, but where the listed values $d_Z$ and $d_X$ are established lower bounds on the relative distances. Of course consulting Section~\ref{sec4.5} we may in some cases conclude that our estimates are sharp,  but we shall not pursue this, except in a few cases. We shall restrict to listing parameters where $d_Z \geq d_X$. 

To demonstrate the advantage of the codes constructed in this section we make two comparisons. Firstly, we compare with the two Gilbert-Varshamov-type existence bounds for asymmetric quantum codes in~\cite{matsumoto2017}[Th.\ 2, Th.\ 4]. As it turns out {\it{none}} of the parameters established in the present section can be foreseen from~\cite{matsumoto2017}. In fact, in most cases we do far better than what is promised by~\cite{matsumoto2017}. We shall not comment further on this, but leave it for the reader to check the details. The second comparison we make follows a tradition started in~\cite{Ezerman2015}[Th.\ 2]. Here, given estimated parameters $[[n,\ell,d_Z/d_X]]_q$ it is being checked from the existence results on linear codes in~\cite{grassl} what one could {\it{possibly}} hope to obtain from the CSS construction using information on the minimum distances of two codes under the {\it{unsubstantiated}}  assumption that the dual of the one is contained in the other. More precisely, we calculate the corresponding numbers $g_1=\delta_Z-d_Z$ and $g_2=\delta_X-d_X$ where $\delta_Z$ and $\delta_X$ are given as follows. To establish $\delta_Z$, and thereby $g_1$ we inspect~\cite{grassl} to see what is  the largest dimension $k_1$ for which a code of length $n$ and minimum distance $d_X$ is guaranteed to exist. Then we define $k_2^\perp=n-(k_1-\ell)$. Finally we inspect~\cite{grassl} once again to see what is the best known minimum distance $\delta_Z$ to exists for a code of dimension $k_2^\perp$
 and being of length $n$. Should the two codes from~\cite{grassl} happen to satisfy that the dual of the one is contained in the other, which we do not have any a priori evidence that is true, then the CSS construction would give us an asymmetric quantum code with estimated parameters $[[n,\ell,\delta_Z/d_X]]_q$. We define $\delta_X$, and thereby $g_2$, in a similar fashion and make similar conclusions regarding the possible existence of a quantum code with parameters $[[n,\ell,d_Z/\delta_X]]_q$.  In conclusion, if for given parameters of an asymmetric quantum code $g_1$ and $g_2$ are not too much larger than $0$ then the parameters of our code can be considered to be good. 

To establish code parameters we use Table~\ref{tabmusigma} in which we collect and reorganize the information from Figure~\ref{figallprime} and Figure~\ref{figalldual}. 
\begin{table}
\caption{Collected information on $\sigma$ and $\mu$}
\begin{center}
\begin{tabular}{c|ccccccccccc}
$i$&1&2&3&4&5&6&7&8&9&10&11\\
\hline
$\sigma(M_i)$&22&19&18&16&15&13&13&12&10&10&9\\
$\mu(M_i)$&1&2&2&3&4&4&3&6&6&6&8\\
\ \\
$i$&12&13&14&15&16&17&18&19&20&21&22\\
\hline
$\sigma(M_i)$&7&7&6&4&5&4&1&3&2&2&1\\
$\mu(M_i)$&8&9&11&12&12&14&15&15&17&18&21
\end{tabular}
\end{center}
\label{tabmusigma}
\end{table}
In Table~\ref{taballasymm} we collect information on $17$ good asymmetric quantum codes, some of which we treat in detail in the following examples.
\begin{table}
\caption{Asymmetric quantum codes over ${\mathbb{F}}_8$ of length $n=22$ coming from the Klein quartic under $\prec_w$}
\begin{center}
\begin{tabular}{ccccccc}
$\ell$&$d_Z$&$d_X$&$C_1$&$C_2$&$g_1$&$g_2$\\
\hline
$1$&$22$&$1$&$E(1)$&$\{ \vec{0} \}$&$0$&$0$\\
$1$&$19$&$2$&$E(2)$&$E(1)$&$0$&$0$\\
$1$ & $15$ & $4$ & $E(5)$ & $E(4)$ & $0$ & $0$\\
$1 $ & $ 9$ & $ 8$ & $E(11) $ & $E(10) $ & $ 0$ & $0 $\\
$2 $ & $ 15$ & $ 3$ & $ E(5)$ & $ E(3)$ & $ 0$ & $0$\\
$2 $ & $ 11$ & $5 $ & $ C(13)$ & $(\tilde{E}(5))^\perp $ & $ 1$ & $1 $\\
$ 3$ & $12 $ & $ 4$ & $ E(8)$ & $ (\tilde{C}(4))^\perp$ & $ 1$ & $ 0$\\
$ 3$ & $ 10$ & $ 6$ & $E(10) $ & $ E(7)$ & $0 $ & $0 $\\
$ 3$ & $ 8$ & $7 $ & $C(10) $ & $C(13) $ & $1 $ & $0 $\\
$4 $ & $13 $ & $3 $ & $E(7) $ & $E(3) $ & $0 $ & $0 $\\
$5 $ & $10 $ & $4 $ & $E(10) $ & $(\tilde{C}(4))^\perp $ & $1 $ & $1 $\\
$7 $ & $ 6$ & $6 $ & $E(14) $ & $ E(7)$ & $1 $ & $1 $\\
$11 $ & $6 $ & $3 $ & $E(14) $ & $E(3) $ & $1 $ & $1 $\\
$12 $ & $ 4$ & $4 $ & $E(17) $ & $(\tilde{C}(4))^\perp $ & $ 1$ & $1 $\\
$14 $ & $4 $ & $ 3$ & $E(17) $ & $E(3) $ & $0 $ & $1 $\\
$15 $ & $3 $ & $ 3$ & $\tilde{E}(3) $ & $E(3) $ & $1 $ & $ 1$\\
$19 $ & $2 $ & $ 2$ & $\tilde{E}(2) $ & $E(1) $ & $0 $ & $0 $
\end{tabular}
\end{center}
\label{taballasymm}
\end{table}
\begin{example}\label{exzero}
Consider $C_2=E(3)$ and $C_1=\tilde{E}(3)$. Then $C_2=C(I,L_2)$ and $C_1=C(I,L_1)$ with $L_2=\{M_1, M_2, M_3\}$ and $L_1=\{M_1,\ldots , M_{17}, M_{19}\}$. Clearly, $L_2 \subsetneq L_1$ ensuring the needed inclusion of codes. We have $\# L_2=3$, $\# L_1=18$ and consequently $\ell = 15$. Further 
$$d(C_1,C_2)\geq \min \{ \sigma (M_4), \ldots , \sigma(M_{17}), \sigma (M_{19})\} = 3$$
$$d(C_2^\perp,C_1^\perp) \geq \min \{\mu(M_4), \ldots , \mu(M_{19})\}=3$$
giving us the parameters $[[22,15,3/3]]_8$ which can be considered good as $g_1=g_2=1$. 
\end{example}
\begin{example}\label{exfirst}
Consider $C_2=(\tilde{C}(4))^\perp$ and $C_1=\tilde{E}(3)$. Then $C_2=C(I,L_2)$ and $C_1=C(I,L_1)$ with $L_2=\{M_1, M_2,M_3,M_4,M_7\}$ and $L_1=\{ M_1,  \ldots , M_{17},M_{19}\}$. By inspection our estimate on $d(C_1,C_2)$ equals the designed distance of $C_1$ which is $3$ and similar for the duals except that the value here is $4$. Replacing $C_1$ with $C_2^\perp$ and vice versa we therefore obtain parameters $[[22,13,4/3]]_8$. However, if instead we choose $C_2=E(3)=C(I,L_2)$ and $C_1=E(17)=C(I, L_1)$ then $L_2=\{M_1, M_2, M_3\}$ is contained in $L_1=\{M_1, \ldots , M_{17}\}$ from which we obtain the better parameters $[[22,14,4/3]]_8$ with corresponding values $g_1=0$ and $g_2=1$. 
\end{example}
In the remaing part of the paper we concentrate on impure codes.
\begin{example}
Let $C_2=E(4)$ and $C_1=E(5)$. Then $d(C_1,C_2)\geq \sigma(M_5)=15$ and $d(C_2^\perp, C_1^\perp)\geq \mu (M_5)=4$ giving us the parameters $[[22,1,15/4]]_8$. However, from Example~\ref{exforimpure} in the previous section we conclude that $d(C_2^\perp)=3$ and therefore the asymmetric quantum code is impure.
\end{example}

\begin{example}
We mention a couple of other examples of impure codes the parameters of which, unfortunately, are not among the best possible of the codes in this section. First let $C_2=C(19)$ and $C_1=C(18)$. Then $d(C_1,C_2)\geq 15$ and $d(C_2^\perp, C_1^\perp)\geq 3$ giving us the parameters $[[22,1,15/3]]_8$. However, $\vec{c}={\mbox{ev}}(X^7-1) \in E(19)=C_2^\perp $ and this word is of Hamming weight only $1$ (see Table~\ref{tabpoints}), meaning that $d(C_2^\perp)=1$. Other examples are $[[22,1,12/5]]_8$ with $d(C_2^\perp)=4$, $[[22,2,17/2]]_8$ with $d(C_2^\perp) =1$, and finally $[[22,2,13/4]]_8$ with $d(C_2^\perp)=3$. 
\end{example}

\section*{Acknowledgment}
The author is grateful to Diego Ruano and Ryutaroh Matsumoto for many fruitful discussions, also in connection with the present paper.

\end{document}